\documentclass[10pt]{article}
\font\smallit=cmti10
\font\smalltt=cmtt10

\usepackage{caption}
\usepackage{color}
\usepackage{amssymb,latexsym,amsmath,epsfig,amsthm} 
\usepackage{empheq}
\usepackage{here}
\usepackage{ascmac}
\makeatletter
\usepackage{here}

\renewcommand\section{\@startsection {section}{1}{\z@}
{-30pt \@plus -1ex \@minus -.2ex}
{2.3ex \@plus.2ex}
{\normalfont\normalsize\bfseries\boldmath}}

\renewcommand\subsection{\@startsection{subsection}{2}{\z@}
{-3.25ex\@plus -1ex \@minus -.2ex}
{1.5ex \@plus .2ex}
{\normalfont\normalsize\bfseries\boldmath}}

\renewcommand{\@seccntformat}[1]{\csname the#1\endcsname. }

\makeatother
\newtheorem{theorem}{Theorem}
\newtheorem{lemma}{Lemma}

\theoremstyle{definition}
\newtheorem{defn}{Definition}[section]
\newtheorem{rem}{Remark}[section]
\newtheorem{exam}{Example}[section]

\begin{document}

\begin{center}
\uppercase{\bf Multi-Dimensional  Chocolate and Nim with a Pass}
\vskip 20pt
{\bf Ryohei Miyadera }\\
{\smallit Keimei Gakuin Junior and High School, Kobe City, Japan}\\
{\tt runnerskg@gmail.com}
\vskip 10pt
{\bf Hikaru Manabe}\\
{\smallit Keimei Gakuin Junior and High School, Kobe City, Japan}\\
{\tt urakihebanam@gmail.com}

\end{center}
\vskip 20pt
\centerline{\smallit Received: , Revised: , Accepted: , Published: } 
\vskip 30pt


\centerline{\bf Abstract}
\noindent
Chocolate-bar games are variants of the CHOMP game. Let $Z_{\geq0}$ be a set of nonnegative numbers and $x,y,z \in Z_{\geq0}$. A three-dimensional chocolate bar comprises a set of cubic boxes sized $1 \times 1 \times 1$, with a bitter cubic box at the bottom of the column at position $(0,0)$. For $u,w \in Z_{\geq0}$ such that $u \leq x$ and $w \leq z$, the height of the column at position $(u,w)$ is $ \min (F(u,w),y) +1$, where $F$ is a monotonically increasing function. We denote this chocolate bar as $CB(F,x,y,z)$. Each player in turn cuts the bar on a plane that is horizontal or vertical along the grooves, and eats the broken piece. The player who manages to leave the opponent with the single bitter cubic box is the winner. In this study, functions $F$ such that the Sprague--Grundy value of $CB(F,x,y,z)$ is $x\oplus y \oplus z$ are characterized.  In a prior work, we characterized function
f for a two-dimensional chocolate-bar game such that the Sprague–Grundy value
of CB(f, y, z) is $y \oplus z$. In this study, we characterize function F such that the
Sprague–Grundy value of  $ CB(F, x, y, z)$ is $ x \oplus y \oplus z$. We also study a multi-dimensional chocolate game, where the dimension is bigger than three, and apply the theory to the problem of pass move in Nim. 

We modify the standard rules of the game to allow a one-time pass, that is, a pass move that may be used at most once in the game and not from a terminal position. Once a pass has been used by either player, it is no longer available. It is well-known that in classical Nim, the introduction of the pass alters the underlying structure of the game, significantly increasing its complexity.

A multi-dimensional chocolate game can show a perspective on the complexity of the game of Nim with a pass.
Therefore, the authors address a longstanding open question in combinatorial game theory. 

The authors present this paper, since the discovery of theirs seems to be significant. It seems to the authors that the relation between chocolate games and Nim with a pass will be an important topic of research soon.

\pagestyle{myheadings} 
\markright{\smalltt INTEGERS:  (  )\hfill} 
\thispagestyle{empty} 
\baselineskip=12.875pt 
\vskip 30pt


\section{Introduction}\label{introductionsection}
Chocolate-bar games are variants of the CHOMP game.
A two-dimensional chocolate bar is a rectangular array of squares in which some of the squares are removed. A poisoned square printed in black is included in some part of the bar. Figure \ref{choco2511} displays an example of a two-dimensional chocolate bar. Each player takes their turn to break the bar in a straight line along the grooves, and eats the broken piece. The player who manages to leave the opponent with the single bitter block (black block) is the winner. 

A three-dimensional chocolate bar is a three-dimensional array of cubes in which a poisoned cubic box printed in black is included in some part of the bar. Figure \ref{3dcho1} displays an example of a three-dimensional chocolate bar. 

Each player takes their turn to cut the bar on a plane that is horizontal or vertical along the grooves, and eats the broken piece. The player who manages to leave the opponent with the single bitter cube is the winner. Examples of cut chocolate bars are depicted in Figures \ref{3dcut1}, \ref{3dcut2}, and \ref{3dcut3}.

\begin{exam}
Here, we provide examples of chocolate bars.\\
\noindent $(i)$
Example of a two-dimensional chocolate bar.
	\vspace{0.1cm}
	
\begin{figure}[H]
\begin{center}
\includegraphics[height=1.3cm]{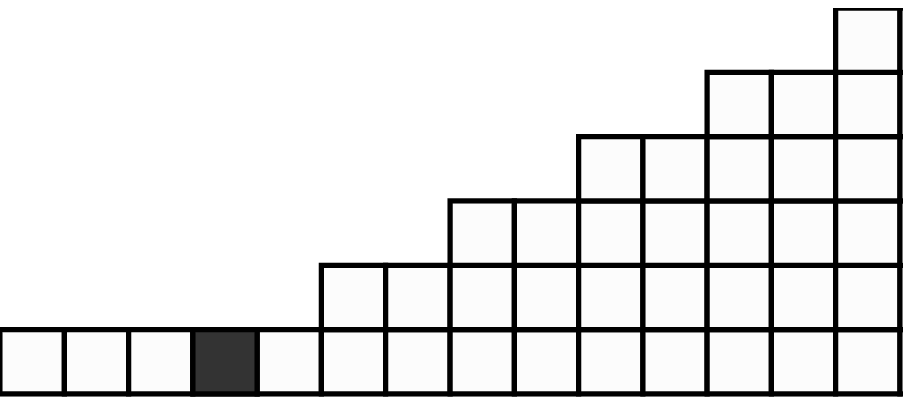}
\captionsetup{labelsep = period}
\caption{ \ }
\label{choco2511}
\end{center}
\end{figure}

\noindent $(ii)$
Example of a three-dimensional chocolate bar.
\begin{figure}[H]
\begin{center}
\includegraphics[height=2.3cm]{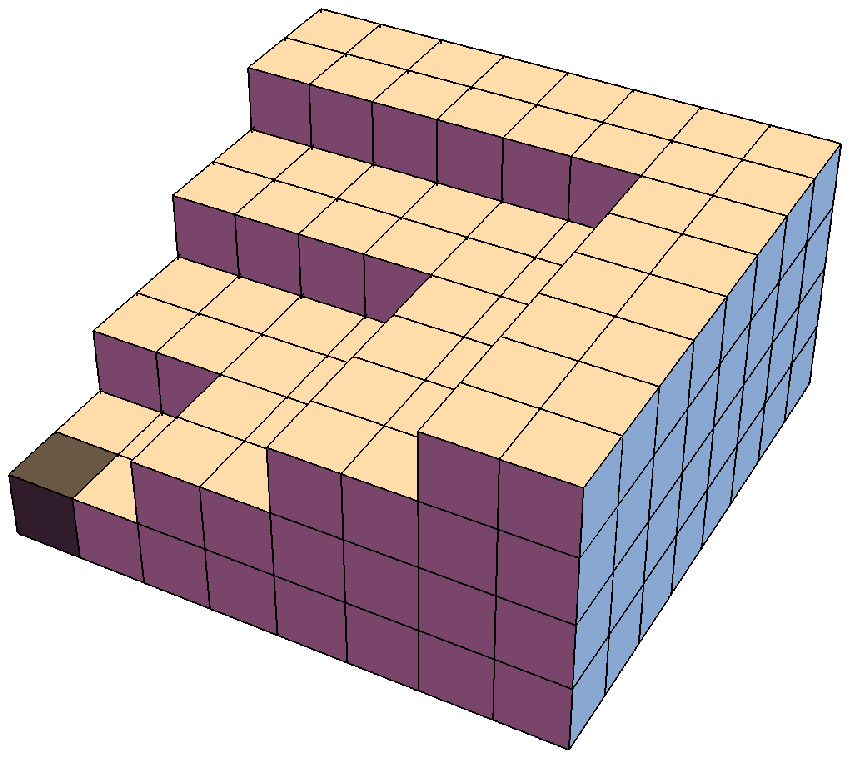}
\captionsetup{labelsep = period}
\caption{ \ }
\label{3dcho1}
\end{center}
\end{figure}
\end{exam}

\begin{exam}
There are three ways to cut a three-dimensional chocolate bar.\\
\noindent
$(i)$ Vertical cut.
\begin{figure}[H]
\begin{center}
\includegraphics[height=4.5cm]{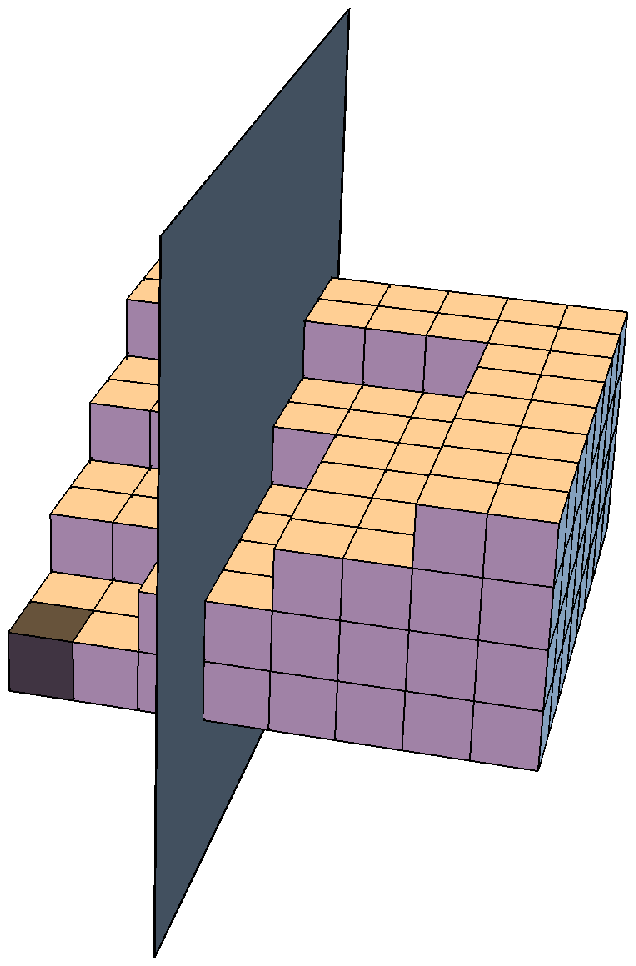}
\captionsetup{labelsep = period}
\caption{ \ }
\label{3dcut1}
\end{center}
\end{figure}

\noindent $(ii)$
Vertical cut.
\begin{figure}[H]
\begin{center}
\includegraphics[height=4.5cm]{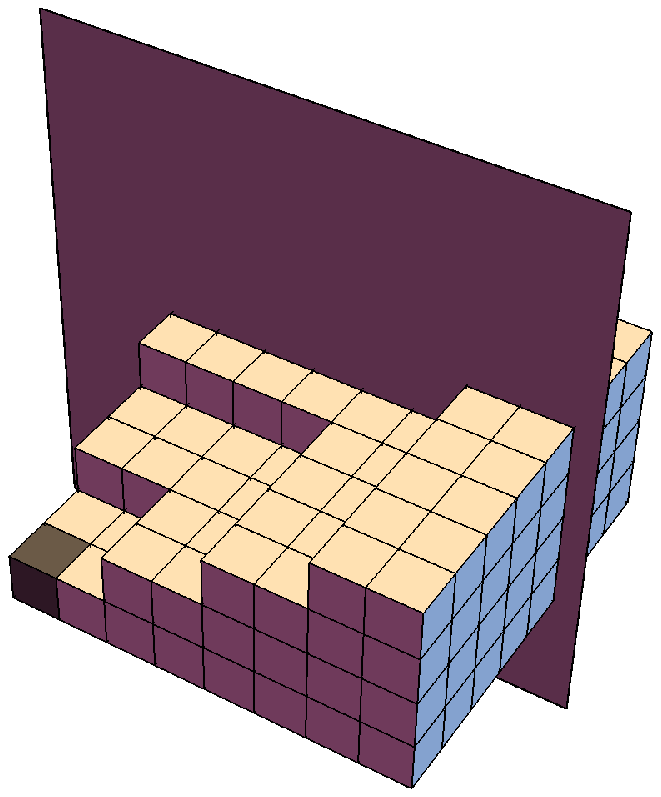}
\captionsetup{labelsep = period}
\caption{ \ }
\label{3dcut2}
\end{center}
\end{figure}

\noindent $(iii)$
Horizontal cut.

\begin{figure}[H]
\begin{center}
\includegraphics[height=2.6cm]{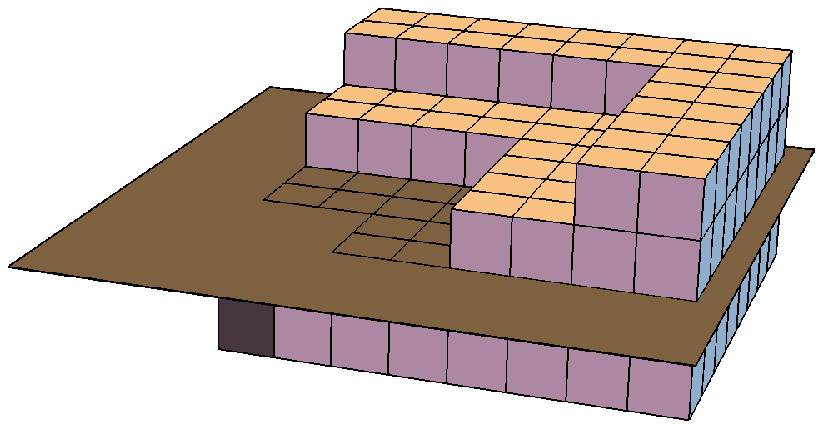}
\captionsetup{labelsep = period}
\caption{ \ }
\label{3dcut3}
\end{center}
\end{figure}

\end{exam}
 The original two-dimensional chocolate bar introduced by Robin \cite{robin} comprises a rectangular bar of chocolate with a bitter corner, as shown in Figure \ref{robinchoco}.
Because the horizontal and vertical grooves are independent, an $m \times n$ rectangular chocolate bar is similarly structured as the game of Nim, which includes heaps of $m-1$ and $n-1$ stones. Therefore, the chocolate-bar game (Figure \ref{robinchoco}) is mathematically the same as Nim, which includes heaps of $5$ and $3$ stones (Figure \ref{nimof5and3}).
Because the Grundy number of the Nim game with heaps of $m-1$ and $n-1$ stones is $(m-1) \oplus (n-1)$, the Grundy number of this $m \times n$ rectangular bar is $(m-1) \oplus (n-1)$.

In addition, Robin \cite{robin} has presented a cubic chocolate bar. For example, see Figure \ref{3dcho1}.
It can be easily determined that the three-dimensional chocolate bar in Figure \ref{3dcho1} is mathematically the same as Nim with heaps of $5$, $3$, and $5$ stones. Hence, the Grundy number of this $6 \times 4 \times 6$ cuboid bar is $5 \oplus 3 \oplus 5$.
\begin{exam}
Here, we provide an example of the traditional Nim game and two examples of chocolate bars.

\begin{figure}[H]
\begin{center}
\includegraphics[height=1.8cm]{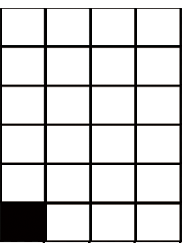}
\captionsetup{labelsep = period}
\caption{ \ }
\label{robinchoco}
\end{center}
\end{figure}

\begin{figure}[H]
\begin{center}
\includegraphics[height=0.9cm]{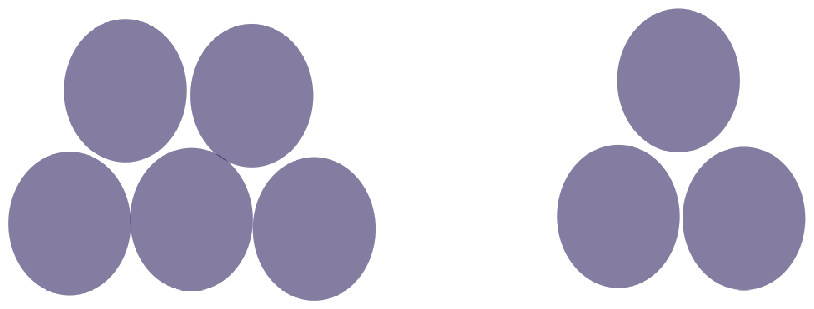}
\captionsetup{labelsep = period}
\caption{ \ }
\label{nimof5and3}
\end{center}
\end{figure}

\begin{figure}[H]
\begin{center}
\includegraphics[height=3cm]{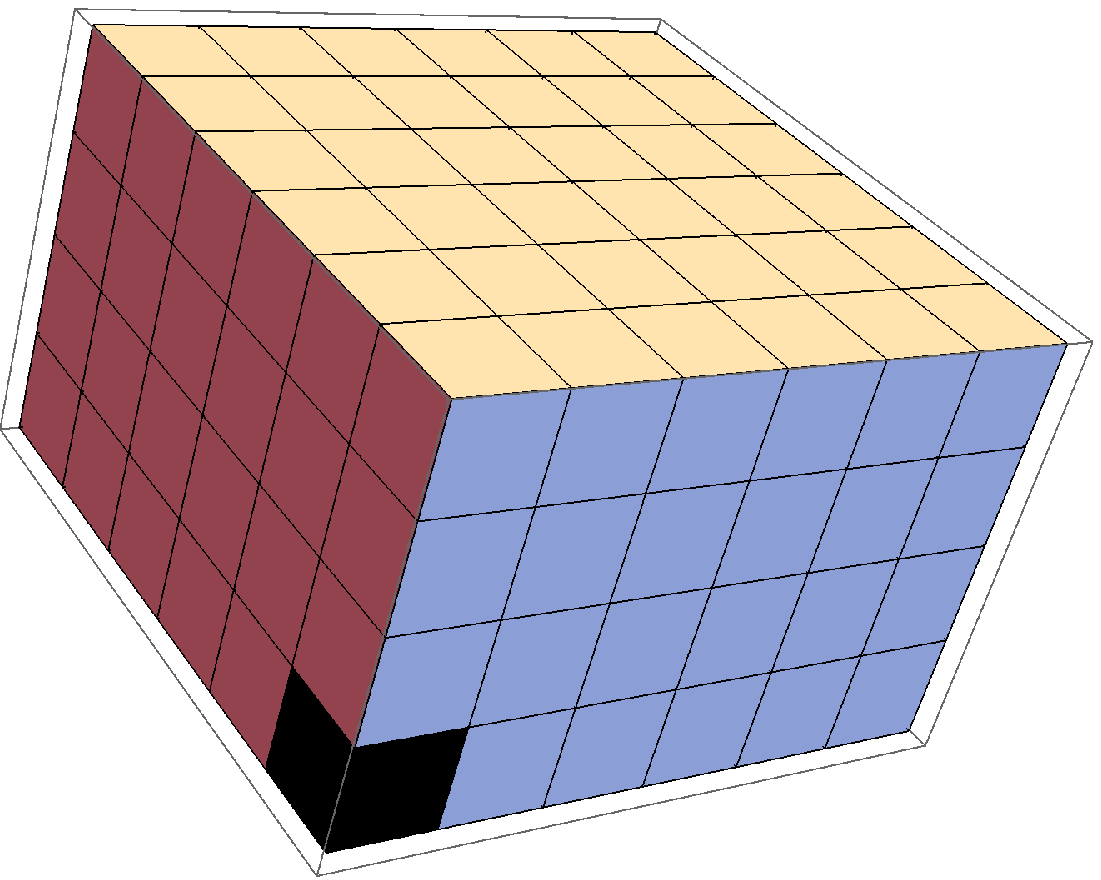}
\captionsetup{labelsep = period}
\caption{ \ }
\label{figurecuboid}
\end{center}
\end{figure}
\end{exam}

Therefore, it is natural to search for a necessary and sufficient condition, wherein a chocolate bar may have a Grundy number calculated using the Nim-sum as the length, height, and width of the bar.

For a two-dimensional chocolate bar, we have already presented the necessary and sufficient condition in \cite{jgame}.

This article aims to answer the following question.\\\\
\noindent
\bf{Question. \ }\normalfont
\textit{What is the necessary and sufficient condition, wherein a three-dimensional chocolate bar may have a Grundy number $(x-1) \oplus (y-1) \oplus (z-1)$, where $x, y$, and $z$ are the length, height, and width of the bar, respectively?}\\

The remainder of this article is organized as follows.
In Section \ref{defandthem}, we briefly review some of the necessary concepts of the combinatorial game theory.

In Section \ref{twodimensionalchoco}, we present a summary of the research results of the two-dimensional chocolate-bar game published in \cite{jgame}, and utilize this result in Section \ref{threedimensionalchoco}.

In Section \ref{threedimensionalchoco}, we study three-dimensional chocolate bars such as the chocolate bars in Figure \ref{3dcho1}, and provide an answer to the above-mentioned research question.
The proof of the sufficient condition for a three-dimensional chocolate bar is straightforward if we utilize the result of the two-dimensional chocolate bar presented in \cite{jgame}; however, 
 the proof of the necessary condition for a three-dimensional chocolate bar is more difficult to obtain, even if we utilize the result presented in \cite{jgame}.
\section{Combinatorial Game Theory Definitions and Theorem }\label{defandthem}
Let $Z_{\geq0}$ be a set of nonnegative integers. 

For completeness, we briefly review some of the necessary concepts of combinatorial game theory; refer to $\cite{lesson}$ or $\cite{combysiegel}$ for more details. 

\begin{defn}\label{definitionfonimsum11}
	Let $x$ and $y$ be nonnegative integers. Expressing them in base 2, 
$x = \sum_{i=0}^n x_i 2^i$ and $y = \sum_{i=0}^n y_i 2^i$ with $x_i,y_i \in \{0,1\}$.
	We define the \textit{nim-sum}, $x \oplus y$, as
	\begin{equation}
		x \oplus y = \sum\limits_{i = 0}^n {{w_i}} {2^i},
	\end{equation}
	where $w_{i}=x_{i}+y_{i} \ (\bmod\ 2)$.
\end{defn}

\begin{lemma}\label{alemmafornimsum}
	Let $x$, $y$, $z \in  Z_{\geq0}$.
	If $y \ne z$, then $x \oplus y \ne x \oplus z. $
\end{lemma}
\begin{proof}
If $x \oplus y = x \oplus z$, then
$y= x \oplus x \oplus y =x \oplus  x \oplus z = z$.
\end{proof}

As chocolate-bar games are impartial games without draws, only two outcome classes are possible.
\begin{defn}\label{NPpositions}
	$(a)$ A position is referred to as a $\mathcal{P}$-\textit{position} if it is a winning position for the previous player (the player who just moved), as long as he/she plays correctly at every stage.\\
	$(b)$ A position is referred to as an $\mathcal{N}$-\textit{position} if it is a winning position for the next player, as long as he/she plays correctly at every stage.
\end{defn}

\begin{defn}\label{sumofgames}
	The \textit{disjunctive sum} of the two games, denoted by $\mathbf{G}+\mathbf{H}$, is a super-game, where a player may move either in $\mathbf{G}$ or $\mathbf{H}$, but not in both.
\end{defn}

\begin{defn}\label{defofmove}
	For any position $\mathbf{p}$ of game $\mathbf{G}$, there is a set of positions that can be reached by precisely one move in $\mathbf{G}$, which we denote as \textit{move}$(\mathbf{p})$. 
\end{defn}

\begin{rem}
Note that \ref{examplecoordinates} and \ref{examofmove11} are examples of a \textit{move}.
\end{rem}

\begin{defn}\label{defofmexgrundy}
	$(i)$ The \textit{minimum excluded value} ($\textit{mex}$) of a set $S$ of nonnegative integers is the least nonnegative integer that is not in S. \\
	 $(ii)$ Let $\mathbf{p}$ be a position of an impartial game. The associated \textit{Grundy number} is denoted by $G(\mathbf{p})$, and is 
 recursively defined by 
	$G(\mathbf{p}) = \textit{mex}\{G(\mathbf{h}): \mathbf{h} \in move(\mathbf{p})\}.$
\end{defn}

\begin{lemma}\label{defofmex2}
	Let $S$ be a set of nonnegative integers and $\textit{mex}(S) = m$ for some $m \in Z_{\geq0}$. Then, $\{k: k < m \text{ and } k \in Z_{\geq0}\} \subset S$. 
\end{lemma}
\begin{proof}
	This also follows directly from Definition \ref{defofmexgrundy}.
\end{proof}

\begin{lemma}\label{grundysmaller}
	If $G(\mathbf{p}) > x$ for some $x \in Z_{\geq0}$, then $\mathbf{h} \in move(\mathbf{p})$ exists,
	such that $G(\mathbf{h}) = x$.
\end{lemma}
\begin{proof}
	This follows directly from Lemma \ref{defofmex2} and Definition \ref{defofmexgrundy}.
\end{proof}

The next result demonstrates the usefulness of the Sprague--Grundy theory in impartial games.
\begin{theorem}\label{theoremofsumg}

Let $\mathbf{G}$ and $\mathbf{H}$ be impartial rulesets, and $G_{\mathbf{G}}$ and $G_{\mathbf{H}}$, respectively, be the Grundy numbers of game $\mathbf{g}$ played under the rules of $\mathbf{G}$ and game $\mathbf{h}$ played under the rules of $\mathbf{H}$. Then, we have the following:\\
	$(i)$ For any position $\mathbf{g}$ of $\mathbf{G}$, 
	$G_{\mathbf{G}}(\mathbf{g})=0$, if and only if $\mathbf{g}$ is a $\mathcal{P}$-position.\\
	$(ii)$ The Grundy number of position $\{\mathbf{g},\mathbf{h}\}$ in game $\mathbf{G}+\mathbf{H}$ is
	$G_{\mathbf{G}}(\mathbf{g})\oplus G_{\mathbf{H}}(\mathbf{h})$.
\end{theorem}
For proof of this theorem, see $\cite{lesson}$.\\

With  Theorem \ref{theoremofsumg}, we can find a $\mathcal{P}$-position by calculating the Grundy numbers and a $\mathcal{P}$-position of the sum of two games by calculating the Grundy numbers of two games.
Therefore, Grundy numbers are an important research topic in combinatorial game theory.
 \section{Two-Dimensional Chocolate Bar}\label{twodimensionalchoco}
Here, the authors define two-dimensional chocolate bars, and present some of their results. Since the operation of cutting and defining Grundy numbers are difficult to understand in the case of three-dimensional bars, the authors present examples \ref{examplecoordinates} and \ref{examofmove11} of two-dimensional chocolate bars.
   The authors  also present the already published Theorem \ref{theoregrundyforf} and a new lemma with a proof as Lemma \ref{lemmaforfxy}. 
   We use Theorem \ref{theoregrundyforf} to prove Lemma \ref{lemmaforfxy} and Theorem \ref{necessarycond3d}  in Section \ref{threedimensionalchoco}. 
Further, we use Lemma \ref{lemmaforfxy} to prove Theorem \ref{sufficientcond3d}  in Section \ref{threedimensionalchoco}.  The employed method involves cutting three-dimensional chocolate bars into sections, and then, applying Theorem \ref{theoregrundyforf} and Lemma \ref{lemmaforfxy} to these sections. Note that a section of three-dimensional chocolate bar is a two-dimensional chocolate bar.   
 
The authors have already determined the necessary and sufficient condition for the Grundy number is 
$(m-1) \oplus (n-1)$ when the width of the chocolate bar monotonically increases with respect to the distance from the bitter square, where $m$ is the maximum width of the chocolate bar, and $n$ is the maximum horizontal distance from the bitter part.
This result has been published in \cite{jgame}, and presented in  Theorem \ref{theoregrundyforf} of this section.
 
\begin{defn}\label{definitionoffunctionf0}
	Function $f$ of $Z_{\geq0}$ into itself is said to be \textit{monotonically increasing} if
	$f(u) \leq f(v)$ for $u,v \in Z_{\geq0}$, with $u \leq v$.
\end{defn} 

\begin{defn}\label{defofbarwithfunc}
	Let $f$ be a monotonically increasing function defined by Definition \ref{definitionoffunctionf0}.
	For $y,z \in Z_{\geq0}$, the chocolate bar has $z+1$ columns, where the 0-th column is the bitter square, and the height of the $i$-th column is
	$t(i) = \min (f(i),y) +1$ for i = 0,1,...,z. We denote this as $CB(f,y,z)$.\\

	Thus, the height of the $i$-th column is determined by the value of $\min (f(i),y) +1$, which is determined by $f$, $i$, and $y$.
\end{defn}

\begin{defn}\label{defofchocogame}
Each player takes their turn to break the bar in a straight line along the grooves into two pieces, and eats the piece without the bitter part. The player who breaks the chocolate bar and eats it, leaving his/her opponent with the single bitter block (black block), is the winner. 
\end{defn}

We fix function $f$ for chocolate bar $CB(f,y,z)$, and call $y,z$ as the coordinates of $CB(f,y,z)$.

\begin{exam}\label{examplecoordinates}
Let $f(t)$  $= \lfloor \frac{t}{2}\rfloor$, where $= \lfloor \ \   \rfloor$ is the floor function.
Here, we present examples of $CB(f,y,z)$-type chocolate bars. Note that function $f$ defines the shape of the bar, and 
the two coordinates, $y$ and $z$, represent the number of grooves above and to the right of the bitter square, respectively. 
Because we use a fixed function $f$, we represent the chocolate-bar positions by coordinates $y,z$.

\begin{figure}[H]
\begin{center}
\includegraphics[height=1.65cm]{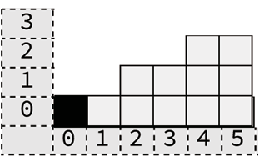}
\caption{$\{2,5\}$}
\label{2yzchoco25} 
\end{center}
\end{figure}
\vspace{0.3cm}

\begin{figure}[H]
\begin{center}
\includegraphics[height=1.65cm]{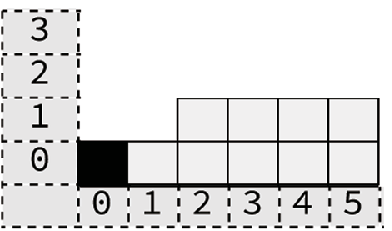}
\caption{$\{1,5\}$}
\label{2yzchoco15}
\end{center}
\end{figure}
\vspace{0.3cm}

\begin{figure}[H]
\begin{center}
\includegraphics[height=1.65cm]{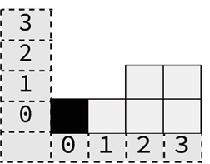}
\caption{$\{1,3\}$}
\label{2yzchoco13}
\end{center}
\end{figure}
\vspace{0.3cm}

\begin{figure}[H]
\begin{center}
\includegraphics[height=1.65cm]{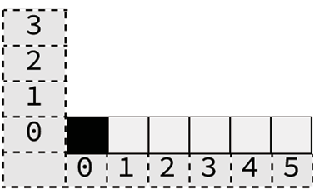}
\caption{$\{0.5\}$}
\label{2yzchoco05}
\end{center}
\end{figure}
\vspace{0.3cm}

\end{exam}
For a fixed function $f$, we define $move_f$ for each position $\{y,z\}$ of the chocolate bar $CB(f,y,z)$. 
Set $move_f(\{y,z\})$ comprises positions of the chocolate bar obtained by cutting the chocolate bar $CB(f,y,z)$ once, and $move_f$ represents a special case of $move$ defined by Definition \ref{defofmove}.
\begin{defn}\label{moveofchocoh}
	For $y,z \in Z_{\ge 0}$, we define \\
	$move_f(\{y,z\})=\{\{v,z \}:v<y \} \cup  \{ \{\min(y, f(w)),w \}:w<z \}$, where $v,w \in Z_{\ge 0}$.
\end{defn}

\begin{rem}
For a fixed function $f$, we use $move(\{y,z\})$ instead of 
$move_f(\{y,z\})$ for convenience. 
\end{rem}

\begin{exam}\label{examofmove11}
Here, we elucidate $move_f$, when $f(t)$  $= \lfloor \frac{t}{2}\rfloor$.
 If we start with position $\{y,z\}=\{2,5\}$ in Figure \ref{2yzchoco25} and reduce $z=5$ to $z=3$, the y-coordinate (first coordinate) will be $\min(2, \lfloor 3/2 \rfloor )=\min(2,1)=1$.
 
 Therefore, we have $\{1,3\} \in move_f(\{ 2,5 \})$; i.e., we obtain $\{1,3\}$ in Figure \ref{2yzchoco13} by cutting $\{ 2,5 \}$. It can be easily determined that

 $\{1,5\}, \{0,5\} \in move_f(\{2,5\})$, $\{1,3\}$

 $  \in move_f(\{1,5\})$, and $\{0,5\} \notin move_f(\{1,3\})$.
 
See Figures \ref{2yzchoco25}, \ref{2yzchoco15}, \ref{2yzchoco13}, and \ref{2yzchoco05}.
\end{exam}

According to Definitions \ref{defofmexgrundy} and \ref{moveofchocoh}, we define the Grundy number of a two-dimensional chocolate bar.

\begin{defn}\label{defofgrundy2d}
	For $y,z \in Z_{\ge 0}$, we define \\
	$\mathcal{G}(\{y,z\})=\textit{mex}(\{\mathcal{G}(\{v,z \}):v<y, v\in Z_{\ge 0}  \} \cup  \{\mathcal{G}(\{ \min(y, f(w)),w \}):w<z,w\in Z_{\ge 0} \})$.
\end{defn}

\begin{defn}\label{definitionoffunctionf}
		Let $h$ be a monotonically increasing function defined by Definition \ref{definitionoffunctionf0}.
Function $h$ is said to have the $NS$ property, if $h$ satisfies condition $(a)$.\\
	$(a)$ Suppose that 
	\begin{equation}
		\lfloor \frac{z}{2^i}\rfloor = \lfloor \frac{z^{\prime}}{2^i}\rfloor \nonumber
	\end{equation}
	for some $z $, $ z^{\prime} \in Z_{\geq 0}$, and some natural number $i$.
	Then, 
	\begin{equation}
		\lfloor \frac{h(z)}{2^{i-1}}\rfloor = \lfloor \frac{h(z^{\prime})}{2^{i-1}}\rfloor. \nonumber
	\end{equation}
\end{defn}

\begin{theorem}\label{theoregrundyforf}
Let $h$ be a monotonically increasing function defined by Definition \ref{definitionoffunctionf}.
 Let $\mathcal{G}_h$ be the Grundy number of $CB(h,y,z)$. Then,  $\mathcal{G}_h(\{y,z\}) = y \oplus z$, 
 if and only if $h$ has the $NS$ property as per Definition \ref{definitionoffunctionf}. 
\end{theorem}
For proof of this theorem, see Theorems 4 and 5 presented in \cite{jgame}.\\\\

The following is a new lemma for two-dimensional chocolate bars, and it will be used for three-dimensional chocolate bars in Section \ref{threedimensionalchoco}. 

\begin{lemma}\label{lemmaforfxy}
Suppose that $h$ has the $NS$ property as per Definition \ref{definitionoffunctionf}, and $y \leq h(z)$ for $y,z \in Z_{\geq 0}$.
Let
\begin{equation}\
	  A =  \{ y \oplus (z-k): k=1,2,\cdots,z \} \nonumber 
\end{equation}
and 
\begin{equation}
	B =  \{\min(y,h(z-k)) \oplus (z-k): k=1,2,\cdots,z  \}. \nonumber 
\end{equation}
Then, $A = B$.
\end{lemma}
This is Lemma 4 of \cite{integer2021}.

Thus far, we have only dealt with two-dimensional chocolate bars for monotonically increasing functions; however, we can similarly consider a three-dimensional chocolate bar $CB(f,y,z)$ for a function, $f$, that is not monotonically increasing by forming a monotonically increasing function $f^{\prime}$, such that
chocolate bars $CB(f,y,z)$ and $CB(f^{\prime},y,z)$ have the same mathematical structure as a game.

For example, the chocolate bar in Figure \ref{choco2511b} is formed by a function that does not monotonically increase, whereas the chocolate bar in Figure \ref{choco2511c} is formed by a monotonically increasing function; however, these two chocolate bars have the same mathematical structure as a game.

\begin{figure}[H]
\begin{center}
\includegraphics[height=1.8cm]{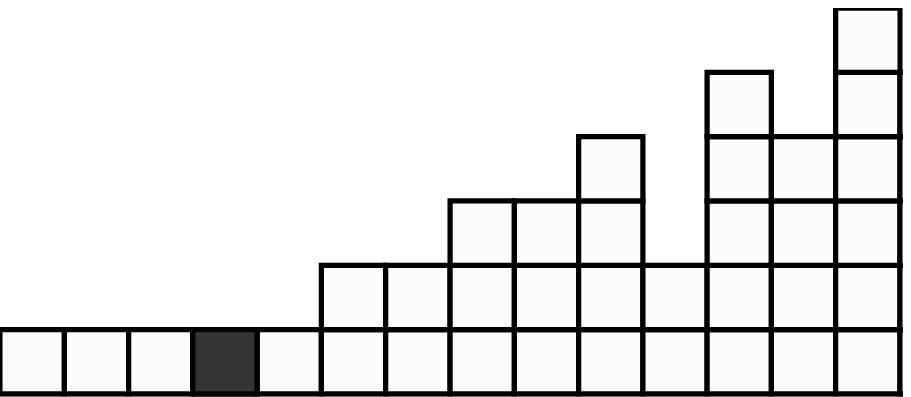}
\captionsetup{labelsep = period}
\caption{ \ }
\label{choco2511b}
\end{center}
\end{figure}
\begin{figure}[H]
\begin{center}
\includegraphics[height=1.8cm]{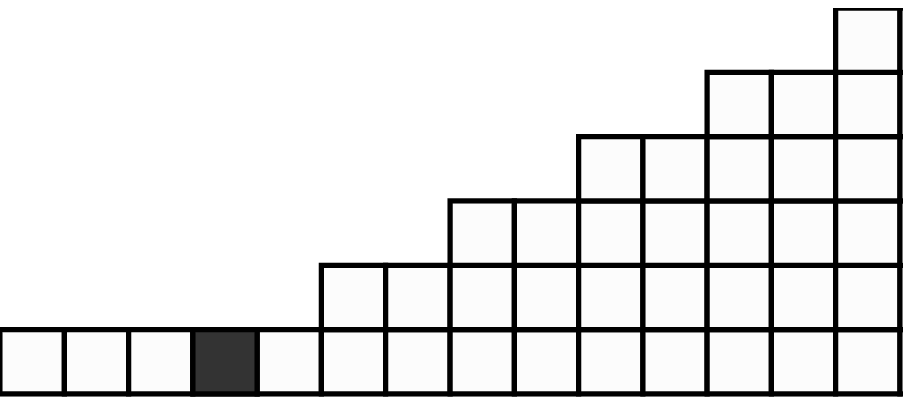}
\captionsetup{labelsep = period}
\caption{ \ }
\label{choco2511c}
\end{center}
\end{figure}
Therefore, it is adequate to study the case of a monotonically increasing function for two-dimensional chocolate bars.
\section{Three-Dimensional Chocolate Bar}\label{threedimensionalchoco}
In this section, we answer the research question that was presented in the \ref{introductionsection} section.
Theorems \ref{sufficientcond3d} and \ref{necessarycond3d} offer proofs for the sufficient and necessary condition, respectively.

\begin{defn}\label{definitionoffunctionf3d}
	Suppose that $F(u,v)\in Z_{\geq0}$ for $u,v \in Z_{\geq0}$. $F$ is said to be \textit{monotonically increasing} if $F(u,v) \leq F(x,z)$ for $x,z,u,v \in Z_{\geq0}$, with $u \leq x$ and $v \leq z$.
\end{defn} 
We generalize Definition \ref{defofbarwithfunc}, and define a three-dimensional chocolate bar.
\begin{defn}\label{defofbarwithfunc3d}
	Let $F$ be the monotonically increasing function in Definition \ref{definitionoffunctionf3d}.\\ 
	Let $x,y,z \in Z_{\geq0}$.
	The three-dimensional chocolate bar comprises a set of $1 \times 1 \times 1$ sized boxes. 
For $u,w \in Z_{\geq0}$, such that $u \leq x$ and $w \leq z$, the height of the column of position $(u,w)$ is $ \min (F(u,w),y) +1$, where $F$ is a monotonically increasing function. 
There is a bitter box in position $(0,0)$.
We denote this chocolate bar as $CB(F,x,y,z)$.\\
\end{defn}

\begin{defn}
We define a three-dimensional chocolate-bar game. Each player takes their turn to cut the bar on a plane that is horizontal or vertical along the grooves, and eats the broken piece. The player who manages to leave the opponent with a single bitter cubic box is the winner. 
\end{defn}

\begin{exam}
Here, we provide an example of a three-dimensional coordinate system and two examples of three-dimensional chocolate bars.\\
\begin{figure}[H]
\begin{center}
\includegraphics[height=2.8cm]{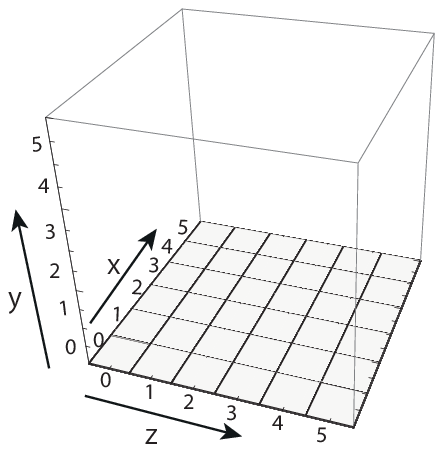}
\captionsetup{labelsep = period}
\caption{ \ }
\label{coordinate3d}
\end{center}
\end{figure}
\vspace{0.3cm}
\begin{figure}[H]
\begin{center}
\includegraphics[height=3.2cm]{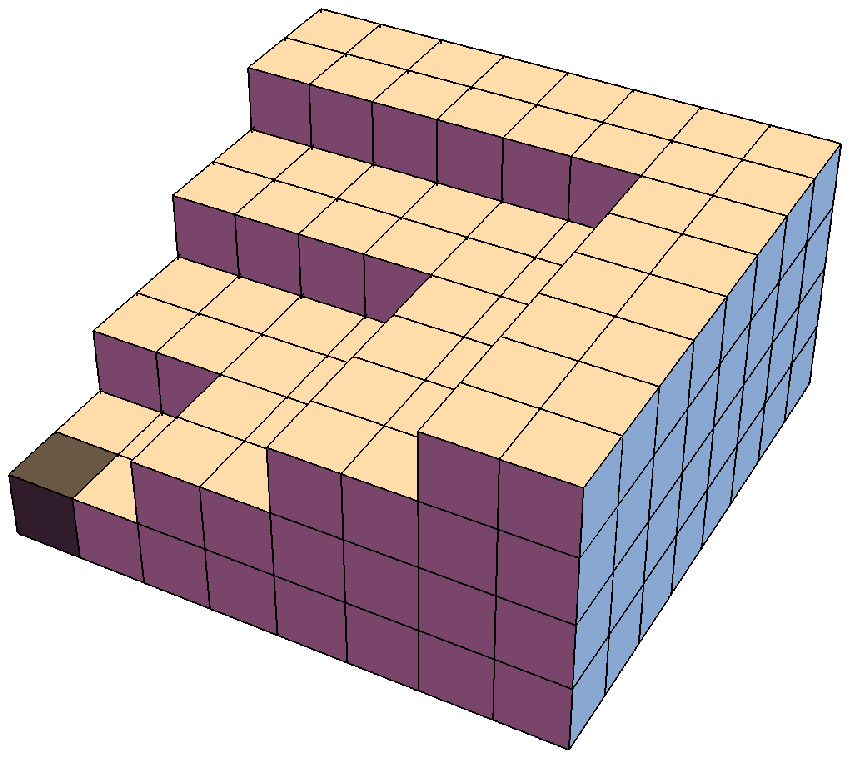}
\caption{$CB(F,7,3,7)$ \\
$F(x,z) $  $= \max(\lfloor \frac{x}{2}\rfloor,\lfloor \frac{z}{2}\rfloor)$.}
\label{3Dchocolate}
\end{center}
\end{figure}
\vspace{0.3cm}
\begin{figure}[H]
\begin{center}
\includegraphics[height=3.2cm]{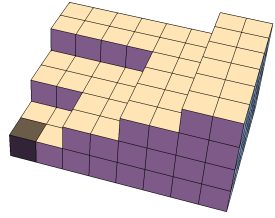}
\caption{$CB(F,5,3,7)$ \\
$F(x,z) $  $= \max(\lfloor \frac{x}{2}\rfloor,\lfloor \frac{z}{2}\rfloor)$.}
\label{3Dchocolate22}
\end{center}
\end{figure}
\end{exam}
Next, we define $move_F(\{x, y, z\})$ in Definition \ref{movefor3dimension}.
Set $move_F(\{x, y, z\})$  contains all the positions that can be reached from position $\{x, y, z\}$ in one step (directly).

\begin{defn}\label{movefor3dimension}
	For $x,y,z \in Z_{\ge 0}$, we define 
\begin{align}
 move_F(\{x,y,z\})= & \{\{u,\min(F(u,z),y),z \}:u<x \} \cup \{\{x,v,z \}:v<y \}   \nonumber \\
 \cup & \{ \{x,\min(y, F(x,w) ),w \}:w<z \}, \text{where $u,v,w \in Z_{\ge 0}$.} \nonumber
\end{align}	

\end{defn}
For example, when $F(x,z) $  $= \max(\lfloor \frac{x}{2}\rfloor,\lfloor \frac{z}{2}\rfloor)$, then $\{5,3,7\} \in move_F(\{7,3,7\})$, because
we obtain the chocolate bar shown in Figure \ref{3Dchocolate22} by reducing the third coordinate of the chocolate bar in Figure \ref{3Dchocolate} from $7$ to $5$.

\begin{rem}
For a fixed function $f$, we use $move(\{x,y,z\})$ instead of 
$move_F(\{x,y,z\})$ for convenience. 
\end{rem}

\begin{lemma}\label{lemmafornimsum3d}
We have the following equation for any $k,h,i \in Z_{\geq 0}$. 
	\begin{align}
		k \oplus h \oplus i= & mex(\{(k-t)\oplus h \oplus i:t=1,2,...,k \}, \\
		 \cup & \{k \oplus (h-t) \oplus i:t=1,2,...,h \} \cup \{k \oplus h \oplus(i-t):t=1,2,...,i \}).\nonumber
	\end{align}
\end{lemma}
\begin{proof}
	We omit the proof because this is a well-known fact regarding Nim-sum $\oplus$. See proposition 1.4. (p.181) in \cite{combysiegel}.
\end{proof}

\begin{theorem}\label{sufficientcond3d}
	Let $F(x,z)$ be a monotonically increasing function.
	Let $g_n(z) = F(n,z)$ and $h_m(x) =F(x,m)$ for $n,m \in Z_{\geq 0}$. 
If $g_n$ and $h_m$ satisfy the $NS$ property in Definition \ref{definitionoffunctionf} for any fixed $n,m \in Z_{\geq 0}$, then the Grundy number of chocolate bar $CB(F,x,y,z)$ is 
	\begin{equation}
		\mathcal{G}(\{x,y,z\}) = x \oplus y \oplus z. \nonumber 
	\end{equation}
\end{theorem}
This is Theorem 3 in \cite{integer2021}

\begin{lemma}
	Let $i \in Z_{\geq0}$ and $z < z^{\prime}$.  We have the following $(a)$ and $(b)$.\\
$(a)$
	\begin{equation}
		\lfloor \frac{z}{2^i}  \rfloor  =   \lfloor \frac{z^{\prime}}{2^i}  \rfloor   \nonumber ,
	\end{equation}
	if and only if $d \in Z_{\geq0}$ exists, such that 
	\begin{equation}
		d \times 2^i \leq z < z^{\prime} < (d+1) \times 2^i.   \nonumber 
	\end{equation}
$(b)$ Suppose that 	
	\begin{equation}
		\lfloor \frac{z}{2^i}  \rfloor  <   \lfloor \frac{z^{\prime}}{2^i} \rfloor.   \label{floorzsmaller}
	\end{equation}
Then, $c, s,t  \in Z_{\geq0}$ exists, such that $s \geq i$, $0 \leq t < 2^s$, and 
	\begin{equation}
	z =	c \times 2^{s+1} +t  < c \times 2^{s+1} + 2^s \leq  z^{\prime}.   \label{ineqfors}
	\end{equation}
\end{lemma}
 This is Lemma 6 in \cite{integer2021}

\begin{theorem}\label{necessarycond3d}
Let $F(x,z)$ be a monotonically increasing function, and let $g_n(z) = F(n,z)$ and $h_m(x) =F(x,m)$ for $n,m \in Z_{\geq 0}$.
Suppose that the Grundy number of chocolate bar $CB(F,x,y,z)$ is
	\begin{equation}
		\mathcal{G}(\{x,y,z\}) = x \oplus y \oplus z.  \nonumber 
	\end{equation}
Then, $g_n$ and $h_m$ satisfy the $NS$ property in Definition \ref{definitionoffunctionf} for any fixed $n,m \in Z_{\geq 0}$.
\end{theorem}
This is Theorem 4 in \cite{integer2021}.

\section{Multi-Dimensional Chocolate Bar}\label{multidimensionalchoco}
In Section \ref{threedimensionalchoco} we studied three-dimensional chocolate game, and in this section we study multi-dimensional chocolate game, where the dimension of chocolate bar is bigger or equal to three.
Let $s \in N$.

\begin{defn}\label{definitionoffunctionfmulti}
	Suppose that $F(x_1,x_2,\cdots, x_s)\in Z_{\geq0}$ for $x_i \in Z_{\geq0}$ for $i=1,2,\cdots, s$. $F$ is said to be \textit{monotonically increasing} if $F(x_1,x_2,\cdots, x_s) $ \\
	$ \leq F((x^{\prime}_1,x^{\prime}_2,\cdots, x^{\prime}_s)$ for $x_i, x^{\prime}_i \in Z_{\geq0}$, with $x_i\leq x^{\prime}_i$ for $i=1,2,\cdots, s$.
\end{defn} 
We define a multi-dimensional chocolate bar.
\begin{defn}\label{defofbarwithfuncmulti}
	Let $F$ be the monotonically increasing function in Definition \ref{definitionoffunctionfmulti}.\\ 
	Let $x_i \in Z_{\geq0}$ for $i=1,2,\cdots, s$.
	The $s+1$ dimensional chocolate bar comprises a set of $1 \times 1 \times 1  \cdots \times 1$ sized $s+1$ dimensional boxes. 
For $x_i \in Z_{\geq0}$, such that $u_i \leq x_i$, the $s+1$-th length of the column of position $(u_1,u_2,\cdots, u_n)$ is $ \min (F(u_1,u_2,\cdots, u_n),y) +1$, where $F$ is a monotonically increasing function. 
There is a bitter box in position $(0,0, \cdots, 0)$.
We denote this chocolate bar as $CB(F,x_1,x_2, \cdots, x_s,y)$.\\
\end{defn}

\begin{defn}
We define a  $s+1$-dimensional chocolate-bar game. Each player takes their turn to cut the bar on a hyper-plane that 
is vertical to the $x_i$-axis
, and eats the broken piece. The player who manages to leave the opponent with a single bitter cubic box is the winner. 
\end{defn}

Next, we define $move_F(\{x_1,x_2,\cdots, x_s, y\})$ in Definition \ref{moveformultiimension}.

Set $move_F(\{x_1,x_2,\cdots, x_s, y\})$  contains all the positions that can be reached from position $\{x_1,x_2,\cdots, x_s, y\}$ in one step (directly).

\begin{defn}\label{moveformultiimension}
	For $x_1,x_2,\cdots, x_s, y \in Z_{\ge 0}$, we define 
\begin{align}
 & move_F(\{x_1,x_2,\cdots, x_s, y\})= \nonumber \\
 & \cup^{i=n} _{i=1}\{\{x_1,x_2,\cdots, x_{i-1},\cdots,x_{i+1}, \cdots, x_s ,\nonumber \\
& \min(F(x_1,x_2,\cdots, x_{i-1},\cdots,x_{i+1}, \cdots, x_s),y) \}:u<x_i \}   \nonumber \\
 \cup & \{ \{x_1,x_2, x_{i},\cdots, x_s,w \}:w<y \}, \text{where $u,v,w \in Z_{\ge 0}$.} \nonumber
\end{align}	

\end{defn}


\begin{rem}
For a fixed function $F$, we use $move(\{x_1,x_2,\cdots, x_s, y\})$ instead of 
$move_F(\{x_1,x_2,\cdots, x_s, y\})$ for convenience. 
\end{rem}

\begin{lemma}\label{lemmafornimsummulti}
We have the following equation for any $k,h,i \in Z_{\geq 0}$. 
	\begin{align}
	&	x_1 \oplus x_2 \cdots  \oplus x_s \nonumber \\
	= & mex(\cup^{s}_{i=1}
		\{(x_1\oplus x_2 \oplus \cdots \oplus x_{i-1} \oplus x_i-k \oplus x_{i+1} \cdots \oplus x_s\}:
		i:k=1,2,...,x_i \}.\nonumber
	\end{align}
\end{lemma}
\begin{proof}
This is a trivial generalization of Lemma \ref{lemmafornimsum3d}.
\end{proof}

\begin{theorem}\label{sufficientcondmulti}
	Let $F(x_1,x_2,\cdots, x_s)$ be a monotonically increasing function.
	Let $g_{x_1,x_2,\cdots, x_{i-1},\cdots,x_{i+1}, \cdots, x_s}(x_i) = F(x_1,x_2, \cdots, x_s)$  for $x_i \in Z_{\geq 0}$. \\
If $g_{x_1,x_2,\cdots, x_{i-1},\cdots,x_{i+1}, \cdots, x_s}$  satisfy the $NS$ property in Definition \ref{definitionoffunctionf} for any fixed ${x_1,x_2,\cdots, x_{i-1},\cdots,x_{i+1}, \cdots, x_s} \in Z_{\geq 0}$, then the Grundy number of chocolate bar $CB(F,x_1,x_2,\cdots, x_s, y)$ is 
	\begin{equation}
		\mathcal{G}(\{x_1,x_2,\cdots, x_s, y\}) = x_1\oplus x_2 \oplus \cdots \oplus x_s. \label{muconclusionnece}
	\end{equation}
\end{theorem}
This is a simple generalization of Theorem \ref{sufficientcond3d}.

\begin{theorem}\label{necessarycondmulti}
Let $F(x_1,x_2,\cdots, x_s)$ be a monotonically increasing function, and let Let $g_{x_1,x_2,\cdots, x_{i-1},\cdots,x_{i+1}, \cdots, x_s}(x_i) = F(x_1,x_2, \cdots, x_s)$  for $x_i \in Z_{\geq 0}$.
Suppose that the Grundy number of chocolate bar $CB(F,x_1,x_2,\cdots, x_s, y)$ is
	\begin{equation}
		\mathcal{G}(\{x_1,x_2,\cdots, x_s, y\}) = x_1\oplus x_2 \oplus \cdots \oplus x_s.  \label{munecessarynimsum}
	\end{equation}
Then, $g_{x_1,x_2,\cdots, x_{i-1},\cdots,x_{i+1}, \cdots, x_s}(x_i)$  satisfies the $NS$ property in Definition \ref{definitionoffunctionf} for any fixed $x_1,x_2,\cdots, x_{i-1},\cdots,x_{i+1}, \cdots, x_s \in Z_{\geq 0}$.
\end{theorem}
This is a simple generalization of Theorem \ref{necessarycond3d}.

\section{Application to the game of Nim with a pass}
We modify the standard rules of the game to allow a one-time pass, that is, a pass move that may be used at most once in the game and not from a terminal position. Once a pass has been used by either player, it is no longer available. It is well-known that in classical Nim, the introduction of the pass alters the underlying structure of the game, significantly increasing its complexity.

A multi-dimensional chocolate game can show a perspective on the complexity of the game of Nim with a pass.
Therefore, the authors address a longstanding open question in combinatorial game theory.

One of the authors has studied this problem in \cite{integers1}. For other research on the game with a pass see
\cite{nimpass}.

\subsection{Two-pile Nim with a pass}

\begin{defn}\label{twopilenim}
There are two piles of stones.
Each player takes their turns, and remove as many stones as she or he likes from one pile.
The player who remove the last stones or a stone is the winner.
Let $t \in Z_{\geq0}$.
We denote by $x$ and $y$ the numbers of stones of the piles.
We assume that a pass move is allowed in this game, but not when $x \leq t$ and $y \leq t$.
\end{defn}

\begin{exam}
The cases of two-pile Nim with a pass are presented in figures \ref{choconimpass}, \ref{choconimpass2}, \ref{choconimpass3}, and \ref{choconimpass4}, where the pass move is denoted by the height. As you see, games of Nim with a pass move are the same as three-dimensional chocolates.
\end{exam}

\begin{figure}[H]
\begin{center}
\includegraphics[height=2.65cm]{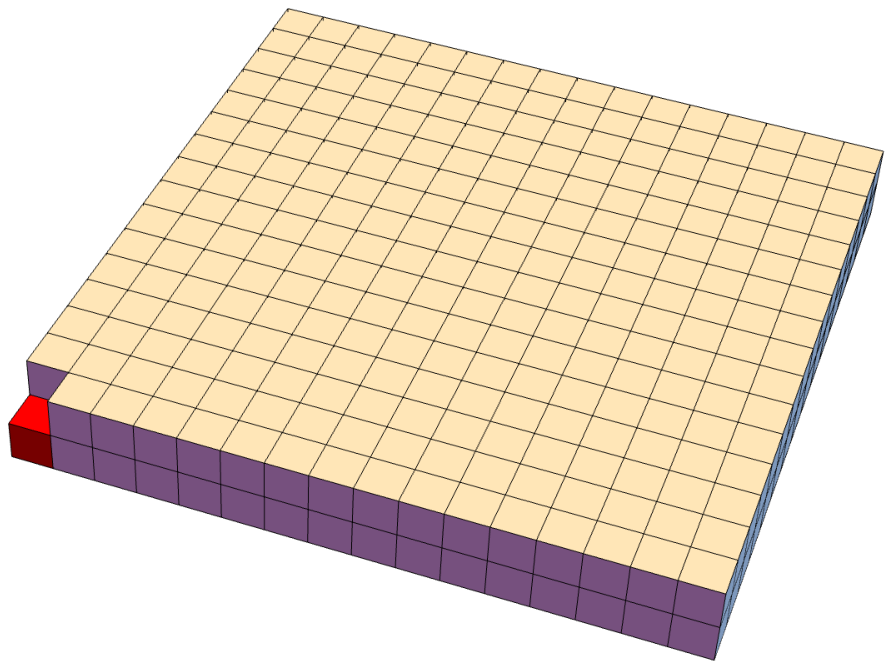}
\caption{$t=1$}
\label{choconimpass} 
\end{center}
\end{figure}
\vspace{0.3cm}

\begin{figure}[H]
\begin{center}
\includegraphics[height=2.65cm]{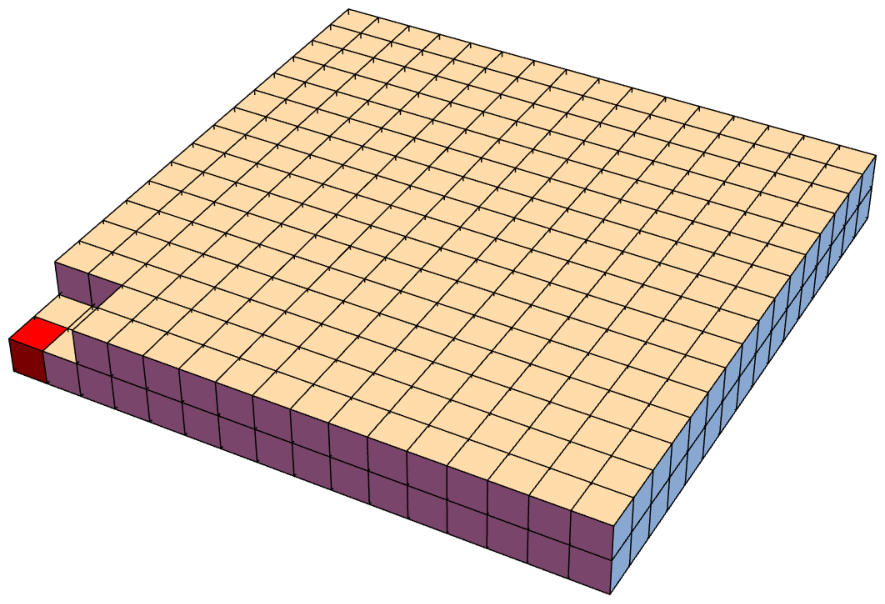}
\caption{$t=2$}
\label{choconimpass2} 
\end{center}
\end{figure}
\vspace{0.3cm}

\begin{figure}[H]
\begin{center}
\includegraphics[height=2.65cm]{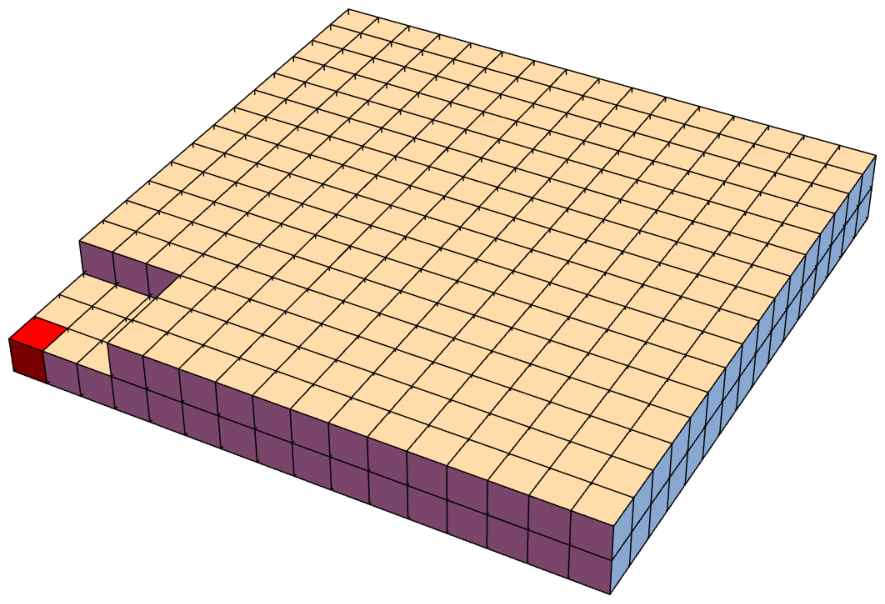}
\caption{$t=3$}
\label{choconimpass3} 
\end{center}
\end{figure}
\vspace{0.3cm}

\begin{figure}[H]
\begin{center}
\includegraphics[height=2.65cm]{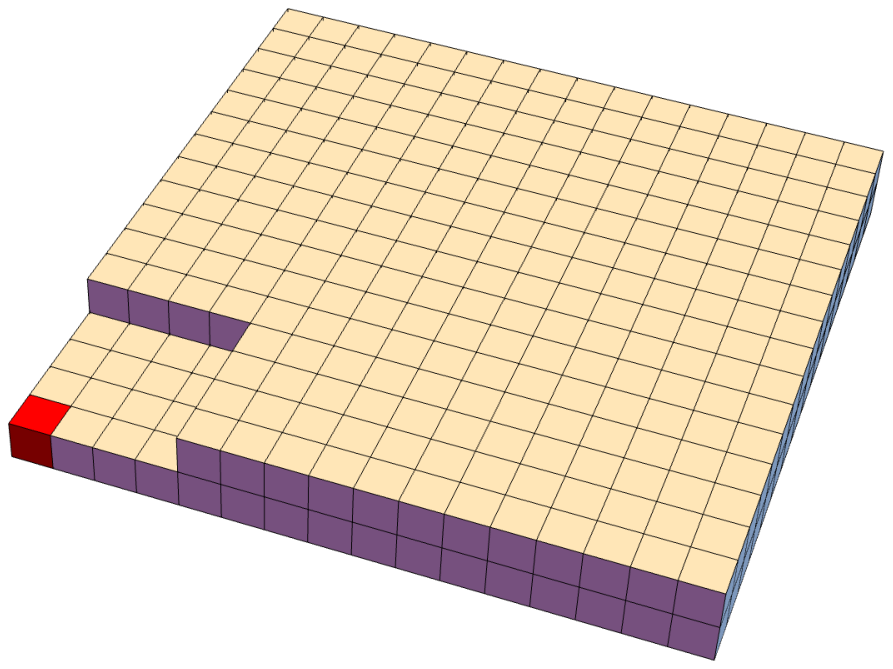}
\caption{$t=4$}
\label{choconimpass4} 
\end{center}
\end{figure}
\vspace{0.3cm}

\begin{theorem}\label{twopiletheorem}
For the game of Nim with a pass of Definition \ref{twopilenim}, $\{x,y,p\}$ is a $\mathcal{P}$-position if and only if
$x \oplus y \oplus p = 0$ only when $t$ is odd.
\end{theorem}
This is direct from theorems \ref{sufficientcondmulti} and \ref{necessarycondmulti}.

\subsection{Three-pile Nim with a pass}

\begin{defn}\label{threepilenim}
There are three piles of stones.
Each player takes their turns, and remove as many stones as she or he likes from one pile.
The player who remove the last stones or a stone is the winner.
Let $t \in Z_{\geq0}$.
We denote by $x$, $y$ and $z$  the numbers of stones of the piles.
We assume that a pass move is allowed in this game, but not when $x \leq t$, $y \leq t$ and $z \leq t$.
\end{defn}

\begin{theorem}\label{threepiletheorem}
For the game of Nim with a pass of Definition \ref{threepilenim}, $\{x,y,z,p\}$ is a $\mathcal{P}$-position if and only if
$x \oplus y \oplus z \oplus p = 0$ only when $t$ is odd.
\end{theorem}
This is direct from theorems \ref{sufficientcondmulti} and \ref{necessarycondmulti}.

\end{document}